\documentclass[11pt,a4paper]{amsart}
\usepackage[T1]{fontenc}
\usepackage[utf8]{inputenc}
\usepackage{hyperref,amssymb,url,upref,verbatim,xspace,mathrsfs}
\RequirePackage[mathscr]{eucal}
\usepackage{mathrsfs}\let\mathcal\mathscr
\def\mainmatter{\renewcommand{\baselinestretch}{1.1}\normalfont}

\makeatletter
\@namedef{subjclassname@2010}{\textup{2010} Mathematics Subject Classification}
\def\l@section{\@tocline{1}{0pt}{0pc}{}{}}
\def\l@subsection{\@tocline{2}{0pt}{1.5pc}{}{}}
\def\l@subsubsection{\@tocline{3}{0pt}{2pc}{}{}}
\makeatother

\AtBeginDocument{%
}

\def\shb{\mathcal{B}}
\def\shd{\mathcal{D}}
\def\shh{\mathcal{H}}
\def\shj{\mathcal{J}}
\def\shl{\mathcal{L}}
\def\shm{\mathcal{M}}
\def\shn{\mathcal{N}}
\def\sho{\mathcal{O}}

\newcommand{\C}{\mathbb{C}}
\newcommand{\N}{\mathbb{N}}
\newcommand{\R}{\mathbb{R}}

\newcommand{\bD}{\boldsymbol{D}}
\newcommand{\Rhom}{R\shh\!om}
\newcommand{\rb}{\mathrm{b}}
\newcommand{\coh}{\mathrm{coh}}
\newcommand{\hol}{\mathrm{hol}}

\newcommand{\rhol}{\mathrm{rhol}}
\newcommand{\Mod}{\mathrm{Mod}}
\newcommand{\cc}{{\C\textup{-c}}}
\newcommand{\rc}{{\R\textup{-c}}}
\newcommand{\XS}{X\times S}
\newcommand{\YS}{Y\times S}
\newcommand{\DXS}{\shd_{\XS/S}}
\newcommand{\DXSa}{\shd_{\XS^*/S^*}}
\newcommand{\pOS}{p^{-1}\sho_S}
\newcommand{\pOSa}{p^{-1}\sho_{S^*}}
\newcommand{\Di}{{}_{\scriptscriptstyle\mathrm{D}}i}

\DeclareMathOperator{\tho}{\mathit{T}\mathcal{H}\mathit{om}}
\DeclareMathOperator{\RH}{RH}
\let\TH\THH
\DeclareMathOperator{\Char}{Char}
\DeclareMathOperator{\rD}{\mathsf{D}}

\DeclareMathOperator{\Hom}{Hom}
\DeclareMathOperator{\Sol}{Sol}
\DeclareMathOperator{\pSol}{{}^\mathrm{p}Sol}
\DeclareMathOperator{\supp}{Supp}
\DeclareMathOperator{\Db}{\mathfrak{Db}}

\let\emptyset\varnothing
\let\setminus\smallsetminus
\let\leq\leqslant
\let\geq\geqslant
\let\oldbullet\bullet
\def\bullet{\raisebox{1pt}{${\scriptstyle\oldbullet}$}}
\def\cf{cf.\kern.3em}
\def\eg{e.g.\kern.3em}
\def\ie{i.e.,\ }
\newcommand{\sfrac}[2]{{#1}/{#2}}

\theoremstyle{plain}
\newtheorem{theorem}{Theorem}[section]
\newtheorem{proposition}[theorem]{Proposition}
\newtheorem{lemma}[theorem]{Lemma}

\theoremstyle{definition}
\newtheorem{example}[theorem]{Example}
\newtheorem{remark}[theorem]{Remark}
\newtheorem{PO}{Property}

\newcommand{\RedefinitSymbole}[1]{%
\expandafter\let\csname old\string#1\endcsname=#1
\let#1=\relax
\newcommand{#1}{\csname old\string#1\endcsname\,}%
}
\RedefinitSymbole{\forall} \RedefinitSymbole{\exists}

\let\ra\rightarrow
\def\to{\mathchoice{\longrightarrow}{\rightarrow}{\rightarrow}{\rightarrow}}
\def\mto{\mathchoice{\longmapsto}{\mapsto}{\mapsto}{\mapsto}}
\def\hto{\mathrel{\lhook\joinrel\to}}

\begin{document}
\author[T. Monteiro Fernandes]{Teresa Monteiro Fernandes}
\address[T. Monteiro Fernandes]{Centro de Matem\'atica e Aplica\c{c}\~{o}es Fundamentais -- Centro de investiga\c c\~ao Operacional e Departamento de Matem\' atica da FCUL, Edif\'icio C 6, Piso 2, Campo Grande, 1700, Lisboa, Portugal}
\email{mtfernandes@fc.ul.pt}

\author[C.~Sabbah]{Claude Sabbah}
\address[C.~Sabbah]{CMLS, École polytechnique, CNRS, Université Paris-Saclay\\
F--91128 Palaiseau cedex\\
France}
\email{Claude.Sabbah@polytechnique.edu}
\urladdr{http://www.math.polytechnique.fr/perso/sabbah}

\thanks{The research of TMF was supported by Funda\c c{\~a}o para a Ci{\^e}ncia e Tecnologia, PEst OE/MAT/\allowbreak UI0209/2011. The research of CS was supported by the grant ANR-13-IS01-0001-01 of the Agence nationale de la recherche.}
\subjclass[2010]{14F10, 32C38, 32S40, 32S60, 35Nxx, 58J10}

\keywords{Relative holonomic $\mathcal D$-module, regularity, constructibility, adjoint functor}

\title{Relative Riemann-Hilbert correspondence in~dimension one}

\begin{abstract}
We prove that, in relative dimension one, the functor $\RH^S$ constructed in a previous work (\cite{MFCS2}) as a right quasi-inverse of the solution functor from the bounded derived category of relative $\mathcal D$-modules with regular holonomic cohomology to that of complexes with relative constructible cohomology satisfies the left quasi-inverse property in a generic sense.
\end{abstract}
\maketitle

\tableofcontents
\mainmatter

\section{Introduction}
Let $p:\XS\to S$ be the projection of a product of complex manifolds~$X$ and~$S$ onto the second factor. The notion of holonomic $\DXS$-modules (or~relative holonomic $\shd$-modules for short) was introduced by the second author in \cite{Sa1} and the notion of relative regular holonomic $\shd$-modules was introduced by the authors in \cite[Def.\,2.1]{MFCS2}. They are the objects of, respectively, the abelian categories $\Mod_{\hol}(\DXS)$ and $\Mod_{\rhol}(\DXS)$.
Recall that relative holonomic modules are coherent modules whose characteristic variety, in the product $(T^*X)\times S$, is contained in $\Lambda\times S$ for some Lagrangian conic closed analytic subset $\Lambda$ of $T^*X$. Regular relative holonomic modules are holonomic modules whose restriction as $\pOS$-modules to the fibers of $p$ have regular holonomic $\shd_X$-modules as cohomologies.

In \cite [Defs.\,2.14\,\&\,2.19]{MFCS} we introduced the notion of relative $\R$- and $\C$\nobreakdash-cons\-truc\-tibility for a complex of sheaves of $\pOS$-modules and proved that the essential image of the functor $\Sol$ on the bounded derived category of relative $\shd$\nobreakdash-modules with holonomic cohomology is contained in that of complexes with relative $\C$\nobreakdash-constructible cohomology.

Under the assumption that $d_ S=1$ ($d$ denotes the dimension of a manifold), we constructed in \cite[\S3.4]{MFCS2} the relative tempered cohomology functors $\TH^S$ and $\RH^S$ by adapting Kashiwara's functors $\TH$ and $\RH$ (\cite{Ka3}, where $\RH$ is denoted by $\Psi$) and proved in \hbox{\cite[Th.\,3]{MFCS2}} that $\RH^S$ is a right quasi inverse for the functor $\pSol:=\Sol[d_ X]$ restric\-ted to the derived category of complexes with relative regular holonomic cohomology. To be more precise, in the absolute case ($S$ is point), $\RH^S(\cdot)$ gives $\RH(\cdot)[d_X]$.

However, contrary to the absolute case, the property of being a left quasi inverse \hbox{remains} open. Indeed, the proof of such a property would require some functorial properties for this category such as stability under proper direct image and inverse image. Although stability under proper direct image holds true, the failure of stability by inverse image remains a main obstruction, in contrast with the absolute case as proved by M.\,Kashiwara in \cite{Ka0}.
In Proposition \ref{P:RH1} we prove that this obstacle can be overcome if $d_ X=1$. More precisely, for each relative holonomic module $\shm$ there exists a discrete set $S_0=S_0(\shm)$ such that, out of $S_0$, for each divisor $Y$ in $X$, the induced system of~$\shm$ along $\YS$ has holonomic cohomologies.

We shall say that a property is satisfied \textit{generically on $S$} if it is satisfied on $\XS^*$, where $S^*$ is the complementary of a discrete subset $S_0$ in $S$. The main purpose of this note is to clarify the natural question arising after \cite{MFCS2}: is $\RH^S$ an equivalence of categories when $d_ X=1$? In other words, does $\RH^S$ also provide in that case a left adjoint to $\pSol$?

The answer is that, any $\shm\in\rD^\rb_{\rhol}(\DXS)$ is isomorphic to $\RH^S(\pSol\shm)$ generically on $S$ by an isomorphism $\Theta(\shm)$, where $\Theta(\bullet)$ satisfies Property \ref{pro:PO} below. According to \cite[Prop.\,6.3]{Ka3} in the absolute case, we have $\Sol\circ \RH\circ \Sol=\Sol$ from $\rD^\rb_{\rhol}(\shd_X)$ to $\rD^\rb_{\cc}(\C_X)$ and, according to \hbox{\cite[Proof of Th.\,1.4.8 \& Th.\,6.1.1]{KKIII}}, $\Sol$ is fully faithful, so that we obtain an isomorphism $\Hom_{\shd_X}(M, \RH(\Sol M))\to \Hom(\Sol M, \Sol M)$. We shall call ``Kashiwara's isomorphism'' the isomorphism $M\to\RH(\Sol M)$ in $\rD^\rb_{\rhol}(\shd_X)$ corresponding to the identity in $\Hom(\Sol M, \Sol M)$.

\begin{PO}\label{pro:PO}
For every $\shm\in\rD^\rb_{\rhol}(\shd_X)$, there exists a discrete subset $S_0\subset S$ depending on $\shm$ only (with $S^*:=S\setminus S_0)$, such that
\begin{enumerate}\renewcommand{\theenumi}{\alph{enumi}}
\item\label{pro:POa}
for $s\in S^*$, we have $Li_s^*\Theta(\shm)=\Theta(Li_s^*\shm)$ and $\Theta(Li_s^*\shm)$ coincides with Kashiwara's isomorphism $Li^*_s\shm\to\RH(\Sol Li^*_s\shm)$,

\item\label{pro:POb}
(\textit{Functoriality in a generic sense}) for a morphism $\tau:\shm\to\shn$ in $\rD^\rb_{\rhol}(\DXS)$, we have
$\RH^S(\tau)\Theta(\shm)\!=\!\Theta(\shn)\tau$ on the open subset \hbox{$X\!\times\!(S\!\setminus\!(S_0(\shm)\cup S_0(\shn))$}.
\end{enumerate}
\end{PO}

Our main result is Theorem \ref{TRH1} in which we prove that, if $d_ X=1$, such a morphism $\Theta(\bullet)$ satisfying \ref{pro:PO} exists. It follows from the second part of \eqref{pro:POa}, together with \cite[Prop.\,1.9]{MFCS2} that, for every $\shm\in\rD^\rb_{\rhol}(\shd_X)$, $\Theta(\shm)$ is an isomorphism on~$S^*$. Let us make precise our claim.

According to \cite[(3.17)]{MFCS2}, we have a natural morphism in $\rD^\rb_{\rhol}(\DXS)$
\begin{equation}\label{E RHS6}
\Phi(\shm): \RH^S(\pSol \shm)\to \Rhom_{\pOS}(\Sol \shm, \sho_{\XS})
\end{equation} with the property that, for $s\in S$ and $M:=Li_s^*\shm$, $Li_s^*\Phi(\shm)=\Phi(M)$ coincides with Kashiwara's morphism (\cf the construction of the functor $\RH$ in \cite{Ka3})\begin{equation}\label{E2}
\tho(\Sol M, \sho_X):=\RH(\Sol M)\to \Rhom (\Sol M, \sho_{X}).
\end{equation}

We prove that, given $\shm\in\rD^\rb_{\rhol}(\shd_X)$, there exists a discrete $S_0$ such that, setting $\shm_*:=\shm_{|\XS^*}$, the natural morphism \begin{multline}\label{E:DRISO}
\Hom_{\DXSa}(\shm_*, \RH^S(\pSol\shm_*))\\
\to\Hom_{\DXSa}(\shm_*, \Rhom_{\pOSa}(\pSol\shm_*,\sho_{\XS^*})[d_ X])\\
\simeq \Hom_{\pOSa}(\pSol \shm_*, \pSol \shm_*)
\end{multline} obtained by applying $\Gamma(X\times S^*, \shh^0(\cdot))$ to the morphism associated to the left composition with $\Phi(\shm_*)$: \begin{multline}\label{E:DRISO2}
\Rhom_{\DXSa}(\shm_*, \RH^S(\pSol\shm_*))\\ \to\Rhom_{\DXSa}(\shm_*, \Rhom_{\pOSa}(\pSol\shm_*,\sho_{\XS^*})[d_ X])\\
\simeq \Rhom_{\pOSa}(\pSol \shm_*, \pSol \shm_*)
\end{multline} is an isomorphism, where for the last isomorphism of \eqref{E:DRISO2} we applied the ``associative law'' relating $\Rhom$ and $\otimes$ (see \cite[A.3\,(b)]{Ka2} and \cite[Ex.\,II.24\,(iii)]{KS}).
We then choose for $\Theta(\shm_*)$ the unique morphism corresponding to the identity in $\Hom_{\pOSa}(\pSol \shm_*, \pSol \shm_*)$. Therefore, in the absolute case ($S$ equal to a point, $\shm$ has regular holonomic cohomology over $\shd_X$), we recover the construction of Kashiwara's isomorphism $\shm\simeq \RH(\Sol\shm)$.

The proof that \eqref{E:DRISO2} is an isomorphism is a consequence of Proposition~\ref{P:reg} below which states that, for any $\shm\in\rD^\rb_{\rhol}(\DXS)$ and any $F\in\rD^\rb_\rc(\pOS)$, there exists a discrete $S_0\subset S$ depending on $\shm$ only such that the natural morphism \begin{multline}\label{E:Theta2}
\Rhom_{\DXSa}(\shm_*, \RH^S(F)_*[-d_X])\\
\to \Rhom_{\DXSa}(\shm_*, \Rhom_{\pOS}(F,\sho_{\XS})_*)
\end{multline}
is an isomorphism, which in turn is a consequence of Proposition \ref{P:RH1} together a comparison result (Lemma \ref{L:elem}). Another consequence of Proposition \ref{P:reg} is the full faithfulness of $\Sol$ in a generic sense (Lemma~\ref{L:ff}).

Throughout this work we assume that $d_X=d_ S=1$.

\subsubsection*{Acknowledgements}
We thank the referees for their valuable comments which helped us to improve the article.

\section{Main results and proofs}

We shall systematically make use of the notation and results in \cite{MFCS} and~\cite{MFCS2}.

\begin{lemma}\label{L:elem}
Let us assume that $X=\C=S$ and let $(x,s)$ be the variables on $\XS$. Let $\shm=\shb_{\{0\}\times S/S}:=\DXS/\DXS x$. Then, setting $S_0=\emptyset$, \eqref{E:Theta2} is an isomorphism for any $F\in\rD^\rb_\rc(\pOS)$.
\end{lemma}

\begin{proof}
Let $Y=\{0\}\subset X$. According to \cite[Prop.\,3.5]{MFCS2}, it is sufficient to consider $F=\C_{U\times S}\otimes \pOS$, for some relatively compact open subanalytic set
$U\subset X$. Thus our goal is to prove
\begin{equation}\label{E:Comp}
\Rhom_{\DXS}(\shm, \Gamma_{\C_{U\times S}}(\Db_{\XS})/\tho(\C_{U\times S}, \Db_{\XS}))\simeq0
\end{equation}
It is sufficient to check \eqref{E:Comp}
for the stalk at any $(x_0, s_0)\in \XS$. If $x_0\neq 0$, the result is trivial since $\supp \shm=\{0\}\times S.$ So we are led to assume $(0, s_0)\in \partial U\times\nobreak S$, since the quotient $\Gamma_{\C_{U\times S}}(\Db_{\XS})/\tho(\C_{U\times S}, \Db_{\XS})_{(0,s)}$ vanishes if $0\notin\partial U$ and the result is again trivial.

Therefore we may assume that $U$ is contained in $X\setminus \{0\}$. We are then allowed to perform a change of generator $u\mapsto x^{-1}u$ since tempered distributions on $U\times S$ are stable by multiplication by $x^{-1}$ and the result follows.
\end{proof}

\begin{proposition}\label{P:RH1}
For any $\shm\in\Mod_{\hol}(\DXS)$ there exists a discrete subset~$S_0$ in $S$ such that, for any reduced divisor $Y$ of~$X$, setting $i_Y:Y\hto X$ the inclusion,
\begin{enumerate}
\item\label{P:RH11}
$\Di_{Y*}\,\Di_Y^*(\shm_*)$ has holonomic cohomologies as an object of the category $\rD^\rb(\DXSa|_{\YS^*})$.
\item\label{P:RH12}
$R\Gamma_{[(X\setminus Y)\times S^*]}(\shm_*)$ is concentrated in degree zero.
\item\label{P:RH13}
$\shm_*(\ast(\YS^*))\simeq\shh^0R\Gamma_{[(X\setminus Y)\times S^*]}(\shm_*)$ is holonomic.
\item\label{P:RH14}
$R\Gamma_{[\YS^*]}(\shm_*)$ has holonomic cohomologies.
\item\label{P:RH15}
If $\shm$ is regular holonomic,
$R\Gamma_{[\YS^*]}(\shm_*)$ and $\shm_*(\ast(\YS^*))$
have regular holonomic cohomologies.
\end{enumerate}
\end{proposition}

\begin{proof}
Let us prove \ref{P:RH1}\eqref{P:RH11}. The question is local on $\XS$, so we can assume that $\shm$ is finitely generated and, by induction on the number of local generators, we may assume that $\shm$ is an holonomic $\DXS$-module with a single generator. Taking coordinates $x$ on $X$ and $s$ on $S$, we are reduced to assuming that $Y=\{x=0\}$ and that $\Char(\shm)\subset(T^*_X X\cup T^*_{\{0\}} X)\times S$. Therefore, there exists a relation
\begin{equation}\label{E:op}
(x\partial_x)^M u=\sum_{j\leq M-1}a_{j}(x,s)\partial_x^ju
\end{equation}
for some non-negative integer $M$ and some holomorphic functions $a_j$ on $\XS$. We~can write $a_j=x^{\ell_j}a'_j$, with $a'_j(0, s)\not\equiv 0$ and $\ell_j\in\N$, and we set $M_0=\max\{0,(j-\ell_j)_{j=0,\dots,M-1}\}$.

Hence, after multiplying by $x^{M_0}$, \eqref{E:op} reads $Pu=0$ for an operator $P=P_0+xQ\in\DXS$ such that $Q$ is of order zero with respect to the $V$\nobreakdash-filtration and
\begin{equation}\label{E:op2}
P_0(s,x,\partial_x)=\sum_{k\leq M}a''_k(s)x^k\partial_x^k
\end{equation}
for some holomorphic functions $a''_k$ on $S$ not all vanishing identically, and it is enough to treat the case of the $\DXS$-module $\DXS/\DXS\cdot P$. Let~$N_0$ be the biggest $k$ such that $a''_k$ does not vanish identically on $S$ (note that~$N_0$ can be~$0$).
Let $S_0$ be the (discrete) zero set of $a''_{N_0}$. Then we are in conditions to apply the relative version of \cite[Th.\,3.3]{LS} to conclude that~$\shm_*$, being elliptic along $\YS^*$, satisfies \ref{P:RH1}\eqref{P:RH11}.

According to the relative versions of Proposition 4.3 in \cite{Ka0} and of
Proposition 7.2.1 of \cite{M}, \ref{P:RH1}\eqref{P:RH11} is equivalent to \ref{P:RH1}\eqref{P:RH14}. On the other hand, $R\Gamma_{[(X\setminus Y)\times S^*]}(\shm_*)$ is concentrated in degree zero since it is the localized module of $\shm_*$ along a divisor. Since $R\Gamma_{[(X\setminus Y)\times S^*]}(\shm_*)$ is the mapping cone of the natural morphism $$R\Gamma_{[\YS^*]}(\shm_*)\to \shm_*$$ we conclude \ref{P:RH1}\eqref{P:RH12} and \ref{P:RH1}\eqref{P:RH13}.

Let us now prove \ref{P:RH1}\eqref{P:RH15}. Let $S_0$ be given by \ref{P:RH1}\eqref{P:RH11}. By \ref{P:RH1}\eqref{P:RH13} and \ref{P:RH1}\eqref{P:RH14}, we have a distinguished triangle in $\rD^\rb_{\hol}(\DXSa)$
$$R\Gamma_{[\YS^*]}(\shm_*)\to \shm_*\to R\Gamma_{[(X\setminus Y)\times S^*]}(\shm_*)\underset{+1}{\to}.$$
Assume that $\shm$ is regular and let $s\in S^*$ be arbitrary. Let us consider the distinguished triangle
\[\tag{$*$}
Li^*_sR\Gamma_{[\YS^*]}(\shm_*)\to Li^*_s\shm_*\to Li^*_sR\Gamma_{[(X\setminus Y)\times S^*]}(\shm_*)\underset{+1}{\to}
\]
The assumption on $\shm$ means that $Li^*_s\shm$ has $\shd_X$-regular holonomic cohomologies. Since $Li^*_s$ commutes with $\Rhom$, we have,
for each $k\in \N $, identifying $X$ to $X\times \{s\}$, a functorial isomorphism in $\rD^\rb(\shd_X)$
$$Li^*_s\Rhom_{\sho_{\XS}}(\sho_{\XS}/x^k\sho_{\XS}, \shm)\simeq \Rhom_{\sho_{X}}(\sho_{X}/x^k\sho_{X}, Li^*_s\shm)$$
and, since $\otimes$ commmutes with $\varinjlim$, we conclude a functorial isomorphism in $\rD^\rb(\shd_X)$ $$Li^*_sR\Gamma_{[\YS]}(\shm)\simeq R\Gamma_{[Y]}(Li^*_s\shm),$$ where the right hand term has regular holonomic cohomologies according to \cite[Th.\,5.4.1]{KKIII}. Therefore $R\Gamma_{[\YS^*]}(\shm_*)$ has regular holonomic cohomologies and the result follows according to the distinguished triangle~$(\ast)$.
\end{proof}

\begin{remark}
Our method in the preceding proof does not extend to the case $d_X>1$ because we do not have in general the analog of \eqref{E:op} and \eqref{E:op2}.
\end{remark}

The following example shows that we cannot avoid the existence of a non\-empty~$S_0$ in Proposition \ref{P:RH1}.

\begin{example}
Let $X=S=\C$, let $Y=\{0\}$ and let $\shm$ be defined by the operator $P(x,s,\partial_x)=x^2\partial _x+g(s)$, where $g$ is a non constant holomorphic function. Then \hbox{$S_0=\{s\in\C\mid g(s)=0\}$}. Let us check that $\shh^{-1}\Di_Y^*(\shm)=0$ and that $\shh^0\Di_Y^*(\shm)$ is not coherent in any neighbourhood of each $(0,s_0)$ such that $g(s_0)=0$. A~local section of the right module $\shb_{\{0\}\times S/S}^{(r)}:=\DXS/x\DXS$ has the form $\sum_{j\leq m}a_j(s)\delta^j(x)$ for some functions $a_{j}(s)$ holomorphic in a neighbourhood of $s_0$, $\delta^j(x)$ denoting the class of $\partial_x^j$ in $\shb_{\{0\}\times S/S}^{(r)}$. That is, in the neighbourhood of any point $(0, s_0)$, $\shb_{\{0\}\times S/S}^{(r)}$ is $\sho_S$-isomorphic to the sheaf $\sho_S[\delta(x)]$, filtered by the degree in $\delta$. We denote by $(\shb_{\{0\}\times S/S}^{(r)})_m$ the $\sho_S$-sub-module of polynomials of degree $\leq m$. The (right) action of $P$ on $\shb_{\{0\}\times S/S}^{(r)}$ is described by
$$\sum_{j\leq m}a_j(s)\delta^j(x)\mto \sum_{j\leq m}((j+1)ja_{j+1}(s)+ a_j(s)g(s))\delta^j(x)$$ In particular $P$ defines a filtered morphism, \ie $$(\shb_{\{0\}\times S/S}^{(r)})_mP\subset(\shb_{\{0\}\times S/S}^{(r)})_m.$$
Let us compute $\ker P=\shh^{-1}\Di_Y^*(\shm)$. Consider a section $u$ of the above form satisfying $uP=0$.
Since by assumption $g$ is non constant and $a_{m+1}=0$, we must have that $a_m=0$ and so henceforward, concluding the vanishing of $\ker P$.

Suppose now that $\sum_{j\leq m} b_j(s)\delta^j(x)=\sum_{j\leq l}a_j(s)\delta^j(x) P,$ with $b_{m}(s)\neq 0.$
Since $b_{m+1}=0$ we have $$(k+1)ka_{k+1}+a_kg=0,\quad \forall k\geq m+1$$ On the other hand we have $a_{k+\ell}=0$ for all $\ell\gg0$, hence, recursively, we conclude that $a_k=0$ for any $k\geq m+1$. Thus $b_m=a_mg$. By descending induction applied to $b_k=(k+1)ka_{k+1}+a_kg$ with $a_{k+1}$ given such that $g$ divides $a_{k+1}$, we find $a_k=b_{k}-(k+1)ka_{k+1}/g$. In particular the condition $uP\in(\shb_{\{0\}\times S/S}^{(r)})_m$ implies that $u\in(\shb_{\{0\}\times S/S}^{(r)})_m$. We conclude that
\[
\mathrm{Coker}P=\varinjlim_m\mathrm{Coker}(P|_{(\shb_{\{0\}\times S/S}^{(r)})_m})\simeq\varinjlim_m (\sho_S/\sho_Sg)^{\oplus m}.
\]
If $g(s_0)\neq 0$, then $g$ is a unit in a neighbourhood of $s_0$, hence the sequence $(\sho_S/\sho_Sg)^{\oplus m}$ is locally zero, which entails that $\shh^0\Di_Y^*(\shm)$ is zero hence coherent in a neighbourhood of $s_0$. If $g(s_0)=0$, the above mentioned sequence is not locally stationary hence
$\shh^0\Di_Y^*(\shm)$ is not coherent in any neighbourhood of $s_0$.
\end{example}

\begin{proposition}\label{P:reg}
Given $\shm\in \rD^\rb_{\rhol}(\DXS)$, there exists a discrete subset~$S_0$ of~$S$ such that \eqref{E:Theta2} is an isomorphism on $\XS^*$ for any $F\in\rD^\rb_{\cc}(\pOS).$
\end{proposition}

\begin{proof}
Our aim is to apply \cite[Lem.\,4.2]{MFCS2} to $\shm(*(\YS^*))$ for a suitable~$S^*$ when $\shm(*(\YS^*))$ is of D-type, in particular when it is strict. We embed~$\shm$ in an exact sequence of regular holonomic modules:
$$0\to \shm_\mathrm{t}\to \shm\to \shm_\mathrm{tf}\to 0$$ where $\shm_\mathrm{t}$ is the submodule of $\sho_S$-torsion germs and $\shm_\mathrm{tf}$ is a strict, \ie an $\sho_S$-flat, module.
\subsubsection*{Step 1}We assume first that $\shm\simeq \shm_\mathrm{t}$. In that case, we claim that we can take for $S_0$ the empty set.

Since $d_ S=1$, the projection of $\supp \shm$ on $S$ is discrete. Given $(x_0,s_0)\in \supp \shm$, we may assume that $\shm$ admits a single generator in a neighbourhood of $(x_0, s_0)$.

Let $(x,s)$ denote a system of local coordinates, $x$ in $X$ and $s$ in $S$, such that $s_0=0\in\C$. We can choose $N\in\N$ such that $s^N\shm=0$ and an easy argument of induction on $N$ allows us to assume $N=1$.

We may then write $\shm$ as a quotient
$$\shm=\DXS/(\DXS\shj+\DXS s),$$
where $\shj$ is a coherent ideal of $\shd_X$ ($X$ identified to $X\times \{0\}$), and the assumption of regularity entails that $\shl:=\shd_X/\shj$ is a regular holonomic $\shd_X$-module.

Moreover, in this local system of coordinates, $\shd_X$ embeds in $\DXS$ as the subsheaf of operators not depending on $s$, so that $\DXS$ is flat over~$\shd_X$ and we have
$\shm':=\DXS/\DXS\shj$ is strict, $\shm'\simeq \DXS\otimes _{\shd_X}\shl$, and $$\shm\simeq p^{-1}(\sho_S/s\sho_S)\otimes_{\pOS}\shm'.$$

According to the ``associative laws'' (\cite[App.\,3, (A.10)]{Ka2}) we get a chain of isomorphisms of functors
\begin{align*}
\Rhom&_{\DXS}(\shm,\bullet)\\
&\simeq \Rhom_{\DXS}\bigl(\DXS\otimes_{\shd_X}\shl, \Rhom_{\pOS}(p^{-1}(\sho_S/s\sho_S)),\bullet\bigl)\\
&\simeq
\Rhom_{\shd_X}\bigl(\shl, \bD'(p^{-1}(\sho_S/s\sho_S))\overset{L}{\otimes} _{\pOS}\bullet\bigl)\\
&\simeq \Rhom_{\shd_X}\bigl(\shl, p^{-1}(\sho_S/s\sho_S)[-1]\overset{L}{\otimes} _{\pOS}\bullet\bigl)\\
&\simeq \Rhom_{\shd_X}\bigl(\shl, Li^*_s(\bullet)[-1]\bigl)
\end{align*}

According to \cite[Prop.\,2.1]{MFCS},
\[
Li^*_s(\Rhom(F,\sho_{\XS}))\simeq \Rhom(Li^*_s F, \sho_X)
\]
and according to \cite[Prop.\,3.26]{MFCS2},
\[
Li^*_s(\RH^S(F))[-1]\simeq \tho(Li^*_s F, \sho_X),
\]
hence the statement follows by \cite[Th.\,6.1.1]{KKIII} with $S_0=\emptyset$.

\subsubsection*{Step 2}
Let us now consider the case where $\shm$ is supported by $\YS$, where~$Y$ is a reduced divisor of $X$. We claim again that the statement holds true with $S_0=\nobreak\emptyset$. By \cite[Th.\,1.5]{MFCS2}, we have $\shm\simeq\Gamma_{[\YS]}(\shm)\simeq \Di_*\shn$, where~$i$ denotes the inclusion $Y\subset X$ and $\shn$ is a direct sum of terms of the form $\shb_{\{y\}\times S/S}\otimes_{\pOS}p^{-1}G_y$, for some $y\in Y$ and some $G_y\in \Mod_{\coh}(\sho_S)$. Hence we may assume that $Y=\{0\}\subset\C$. Taking a local coordinate $x$ on $\C$ vanishing on~$Y$, we are reduced to proving that~\eqref{E:Theta2} applied to
\[
\shm=\shb_{\{0\}\times S/S}\otimes_{\pOS}(\C_{\{0\}\times S}\otimes p^{-1}G)
\]
is an isomorphism when $G$ is a coherent $\sho_S$-module. By the ``associative laws'' above mentioned
this amounts to checking the same property for the regular holonomic module $\shb_{\{0\}\times S/S}$ which in turn follows by Lemma \ref{L:elem}.

\subsubsection*{Step 3} Let us now assume that $\shm$ is strict. If $\shm$ is a locally free $\sho_{\XS}$-module, the assertion follows from \cite[Lem.\,3.17]{MFCS2}. Otherwise the natural stratification associated to $\shm$ is $\{X\setminus Y, Y\}$ for some reduced divisor $Y$ in~$X$. Let $S_0$ be determined by Proposition \ref{P:RH1} and let us embed $\shm_*$ in an exact sequence where, according to Proposition \ref{P:RH1}, $\shm_*(\ast(\YS^*))$ and $\shh^1_{[\YS^*]}(\shm_*)$ are regular holonomic modules:
\[
\tag{$\ast$}
0\ra\Gamma_{[\YS^*]}(\shm_*)\to\shm_*\\
\to \shm_*(\ast(\YS^*))\to \shh^1_{[\YS^*]}(\shm_*)\ra 0.
\]
Since the functor of localization is exact, strictness is preserved by localization, hence $\shm_*(\ast(Y\!\times\!S^*))$ is of D-type along $\YS^*$ in the sense of \cite[Def.\,2.10] {MFCS2}. Therefore \cite[Lem.\,4.2]{MFCS2} gives the statement for $\shm_*(\ast(\YS^*))$.

We apply our previous results to the following exact sequences of regular holonomic $\DXS$-modules obtained by splitting $(\ast)$:\enlargethispage{1.5\baselineskip}%
\begin{gather}\label{E:loc1}
0\to\Gamma_{[\YS^*]}(\shm_*)\to\shm_*\to \sfrac{\shm_*}{\Gamma_{[\YS^*]}\shm_*}\to 0\\[3pt]
\label{E:loc2}
0\to \sfrac{\shm_*}{\Gamma_{[\YS^*]}\shm_*}\to\shm_*(\ast(\YS^*))\to \shh^1R\Gamma_{[\YS^*]}\shm_*\to 0.
\end{gather}
By Step 2 applied to \eqref{E:loc2} we conclude that the statement holds true for $\sfrac{\shm_*}{\Gamma_{[\YS^*]}\shm_*}$. From \eqref{E:loc1} and Step 2 we conclude that the statement holds for~$\shm_*$.
\end{proof}

\begin{theorem}\label{TRH1}
For each $\shm\in\rD^\rb_{\rhol}(\DXS)$ there exists a discrete set $S_0\subset S$ and an isomorphism in $\rD^\rb_{\rhol}(\DXSa)$ satisfying \ref{pro:PO}:
$$\Theta(\shm_*):\shm_*\to\RH^S(\pSol\shm)_*.$$
\end{theorem}

\begin{proof}
The existence of the morphism $\Theta(\shm_*)$ is an immediate consequence of Proposition \ref{P:reg} as explained in the introduction. Moreover, according to \cite[Prop.\,2.1]{MFCS} and \cite[Prop.\,3.29]{MFCS2}, for any $s\in S$, the derived functor $Li^*_s$ commutes with all the derived functors appearing in \eqref{E:DRISO2}, so \eqref{pro:POa} of Property \ref{pro:PO} holds. As explained in the introduction, this implies that $\Theta(\shm_*)$ is an isomorphism.

Let us prove that $\Theta$ satisfies the generic functoriality property \eqref{pro:POb} of \ref{pro:PO}. Let $\tau:\shm\to\shn$ be a morphism in $\rD^\rb_{\rhol}(\DXS)$. In order to prove that, generically on~$S$, we have $\RH^S(\tau_*)\Theta(\shm_*)=\Theta(\shn_*)\tau_*$, it suffices to prove this equality after applying $Li^*_s$ for any $s\in S^*$, according to \cite[Prop.\,1.9]{MFCS2}. Now, from \eqref{pro:POa} of Property~\ref{pro:PO} already proved and from \cite[Prop.\,3.29]{MFCS2}, we are reduced to proving the result for $Li^*_s\tau:Li^*_s\shm\to Li^*_s\shn$, that is, we are reduced to the absolute case. The statement then follows by the functoriality of $\Theta$ in the absolute case.
\end{proof}

\begin{lemma}\label{L:ff}
The functor $\Sol$ is fully faithful in $\rD^\rb_{\rhol}(\DXS)$ in a generic sense, that is, the functor $\Sol$ induces generically a bifunctorial isomorphism
$$\Rhom_{\DXS}(\shm, \shn)\to \Rhom_{\pOS}(\Sol\shn, \Sol\shm)$$ for any $\shm, \shn\in\rD^\rb_{\rhol}(\DXS)$.
\end{lemma}

\begin{proof} Let us start by constructing the morphism. It is obtained as the following composition, where we use the isomorphism $\Theta(\shn_*)$ given by Theorem \ref{TRH1} above:
\begin{multline*}
\Rhom_{\DXSa}(\shm_*, \shn_*)\simeq\Rhom_{\DXSa}(\shm_*, \RH^ S(\pSol \shn)_*)\\
\shoveright{\underset{(a)}{\to}\Rhom_{\DXSa}(\shm_*, \Rhom_{\pOS}(\Sol\shn, \sho_{X\times S})_*)}\\[-4pt]
\simeq\Rhom_{\pOSa}(\Sol \shn_*, \Sol \shm_*),
\end{multline*}
where $(a)$ is induced by $\Phi(\shn)$ (\cf\eqref{E RHS6}) and the last isomorphism follows by the ``associative properties'' and \cite[Ex.\,II.24\,(iii)]{KS}. According to Proposition \ref{P:reg}, the morphism $(a)$ is an isomorphism and the result follows.
\end{proof}

\begin{remark}
If $\shm$ is regular holonomic and admits locally a single generator~$u$ such that $\shj:=\{P\in\DXS\mid Pu=0\}$ is monogenic, it is easy to verify that the associated discrete set $S_0\subset S$ mentioned in Proposition \ref{P:RH1} can be taken to be the empty set:

The assumption on $\Char(\shm)$ entails that we can choose as a generator of~$\shj$ an operator $P$ of the form
$$P(x,s,\partial_x)=x^j\partial_x^m+\sum_{k<m} a_{k}(x,s)\partial^k _x$$
The assumption of regularity means that, for each fixed $s$, $P(x,s,\partial_x)$ has a regular singularity in $x=0$ as a section of $\shd_X$. Hence each coefficient $a_{k}(x,s)$ has a zero of order at least $j-m+k$ at $x=0$. It follows that the coefficient $a''_{N_0}$ in the proof of Proposition \ref{P:RH1} is equal to $1$ (it is the coefficient of the term $x^j\partial_x^m$).

However we cannot generalize this result to arbitrary regular holonomic modules because, contrary to the absolute case, we do not have the tools to perform a devissage.
\end{remark}

\providecommand{\eprint}[1]{\href{http://arxiv.org/abs/#1}{\texttt{arXiv\string:\allowbreak#1}}}


\begin{thebibliography}{1}

\bibitem{Ka0}
M.~Kashiwara, \emph{{On the holonomic systems of differential equations II}},
Invent. Math. \textbf{49} (1978), 121--135.

\bibitem{Ka3}
M.~Kashiwara, \emph{The {Riemann-Hilbert} problem for holonomic systems}, Publ.
RIMS, Kyoto Univ. \textbf{20} (1984), 319--365.

\bibitem{Ka2}
M.~Kashiwara, \emph{{$D$}-modules and microlocal calculus}, Translations of
Mathematical Monographs, vol. 217, American Mathematical Society, Providence,
R.I., 2003.

\bibitem{KS}
M.~Kashiwara and P.~Schapira, \emph{Sheaves on manifolds}, Grundlehren Math. Wiss, \textbf{217}, Springer Verlag, (1990).

\bibitem{KKIII}
M.~Kashiwara and T.~Kawai, \emph{{On the holonomic systems of differential
equations (systems with regular singularities) III}}, Publ. RIMS, Kyoto Univ.
\textbf{17} (1981), 813--979.

\bibitem{LS}
Y.~Laurent and P.~Schapira, \emph{Images inverses des modules différentiels},
Compositio Math. \textbf{61} (1987), no.~2, 229--251.

\bibitem{M}
Z.~Mebkhout, \emph{{Le formalisme des six opérations de Grothendieck pour les
{$\mathcal{D}$}-modules cohérents}}, Travaux en cours, vol.~35, Hermann,
Paris, 1989.

\bibitem{MFCS}
T.~Monteiro~Fernandes and C.~Sabbah, \emph{{On the de Rham complex of mixed
twistor $\mathcal{D}$-modules}}, Internat. Math. Res. Notices (2013), no.~21,
4961--4984.

\bibitem{MFCS2}
T.~Monteiro~Fernandes and C.~Sabbah, \emph{{Riemann-Hilbert correspondence for mixed twistor
$\mathcal{D}$-modules}}, to appear in Journal of the Institute of Mathematics of Jussieu, first publ.\ online 19/5/2017, \href{http://dx.doi.org/10.1017/S1474748017000184}{\texttt{doi:10.1017/S1474748017000184}}.

\bibitem{Sa1}
C.~Sabbah, \emph{{Polarizable twistor $\mathcal{D}$-modules}}, Ast{\'e}risque,
vol. 300, Soci{\'e}t{\'e} Math{\'e}matique de France, Paris, 2005.

\end{thebibliography}
\end{document}